\documentclass[11pt]{article}
\usepackage{hyperref,etex,amsfonts,amssymb,amsmath,
amsthm,amscd,euscript,
array,mathrsfs}

\title{$C^*$-algebras and direct integral decomposition \\for Lie supergroups 
}
\author{Karl-Hermann Neeb\footnote{Department Mathematik, Friedrich-Alexander Universit\"{a}t Erlangen--N\"{u}rnberg, \texttt{neeb@mi.uni-erlangen.de}} , 
Hadi Salmasian\footnote{Department of Mathematics and Statistics, University of Ottawa,
\texttt{hsalmasi@uottawa.ca}
}
}
\input{mymacros-3.sty}
\newcommand{\fn}{{\mathfrak n}}
\newcommand{\fg}{{\mathfrak g}}

\begin{document}

\maketitle

\begin{abstract}
For every finite dimensional Lie supergroup $(G,\g g)$, we define a $C^*$-algebra $\cA:=\cA(G,\g g)$, and show that there exists a canonical bijective correspondence between unitary representations of $(G,\g g)$ and nondegenerate $*$-representations of $\cA$. The proof of existence of such a correspondence relies on a subtle characterization of smoothing operators of unitary representations from \cite{nsz}. 

For a broad class of Lie supergroups, 
which includes nilpotent as well as classical simple ones,
we  prove that the associated $C^*$-algebra is CCR. In particular, we obtain the uniqueness of direct integral decomposition for unitary representations of these Lie supergroups.
\end{abstract}

\section{Introduction}

Unitary representations of Lie supergroups 
play an important role in the mathematical theory of SUSY quantum mechanics.  
One distinguished example of the role of these unitary representations is  the classification of free relativistic super particles (see \cite{FSZ81} and \cite{SS74}), where a super analogue of the little group method of Mackey and Wigner is used.

Although the super version of the Mackey--Wigner method was used in the physics literature as early as the 1970's,  the problem of mathematical validity of this method in the context of supergroups was not addressed until less than a decade ago.
This was done in \cite{varadarajan}, where the authors remedy this issue by laying the mathematically rigorous foundations of the analytic theory of unitary representations of Lie supergroups, using the equivalence of categories between the category of Lie supergroups and the category of  
\emph{Harish--Chandra pairs}
\cite[Sec.\ 3.8]{delignemorgan}, \cite[Sec.\ 3.2]{kostant}. 
The Harish--Chandra pair description of Lie supergroups will be explained  in Definition \ref{def-blsupergroup} below.

The groundwork laid in \cite{varadarajan} has spawned 
research on the harmonic analysis of Lie supergroups. In particular, in \cite{SalCMP}
the irreducible unitary representations of a nilpotent Lie supergroup  are classified using an extension of Kirillov's orbit method (see also \cite[Sec.\ 8]{nsLNM}). For Lie supergroups corresponding to basic classical  Lie superalgebras \cite[Def.\ 1.14]{musson}, the irreducible unitary representations are indeed highest weight modules
\cite[Sec.\ 7]{nsLNM}, 
and therefore they are completely classified by the work done in
\cite{jakobsen}.

 The  goal of this paper is to systematically study  disintegration of arbitrary unitary representations of Lie supergroups into
direct integrals of  
  irreducible representations. To this end, for every finite dimensional Lie supergroup $(G,\g g)$ we construct a $C^*$-algebra $\cA:=\cA(G,\g g)$ whose nondegenerate $*$-representations are in bijective correspondence with unitary representations of $(G,\g g)$.  
The $C^*$-algebra $\cA$ is obtained as the completion 
(with respect to a suitable $C^*$-seminorm) of
a crossed product $*$-algebra $\cA^\circ$  associated to the adjoint action of $G$ 
on the universal enveloping algebra $\bfU(\g g)$.
Here indeed it will be more convenient to replace $G$ by a slightly larger group $G_\boldeps\cong G\times \{1,\boldeps\}$, as the action of the extra element $\boldeps$ will automatically keep track of the  $\Z_2$-grading of the representation space.
Starting from a unitary representation $(\pi,\rho^\pi,\ccH)$ 
of $(G,\g g)$,  we obtain a representation of $\cA$ by first extending 
$(\pi,\rho^\pi,\ccH)$ canonically to  $\cA^\circ$, and then uniquely to a nondegenerate $*$-representation $(\whpi,\ccH)$ of $\cA$ by continuity. Nevertheless, the construction of a representation of $(G,\g g)$ from a representation $(\whpi,\ccH)$ of $\cA$ is more subtle, 
because the standard method of extending $(\whpi,\ccH)$ to the multiplier algebra 
$M(\cA)$ is \emph{not} sufficient in our context. 
Indeed the Lie superalgebra $\g g$ does not act on $\cA$ by multipliers. To circumvent this issue, we  use the extension of $(\whpi,\ccH)$  to the multiplier algebra $M(\cA^\circ)$, and use the fact that $G$ and $\g g$ act on $\cA^\circ$ through $M(\cA^\circ)$. To complete the construction of the unitary representation of $(G,\g g)$, we need to show that
the action of $\g g$ is indeed defined on $\ccH^\infty$. To this end,  we prove that $\whpi(\cA^\circ)\ccH=\ccH^\infty$, where $\ccH^\infty$ denotes the space of smooth vectors of the action of $G$ on $\ccH$.
The proof of the latter statement requires the Dixmier--Malliavin Theorem \cite{dixmal} and a subtle result
from \cite[Thm 2.11]{nsz}
on the characterization of smoothing operators of unitary representations, that is,
 operators $A:\ccH\to\ccH$ which map $\ccH$ into $\ccH^\infty$.

By the standard machinery of $C^*$-algebras \cite{dixmier}, statements on the existence and uniqueness of disintegration of nondegenerate $*$-representations of $\cA$ can be transformed to similar statements on direct integral decompositions of unitary representations of $(G,\g g)$. 
To obtain  uniqueness of disintegration, 
 it suffices  to know that $\cA$ is 
 CCR, that is, the image of every irreducible $*$-representation of 
  $\cA$ lies in the algebra of compact operators.
(Such $C^*$-algebras are sometimes called \emph{liminal}.)
We prove that $\cA$ is  CCR for a broad class of Lie 
supergroups, which includes  nilpotent Lie supergroups as well as those which correspond to classical simple Lie superalgebras (see \cite[Sec.\ 1.3]{musson}). Therefore for the aforementioned classes of Lie supergroups, one obtains uniqueness of disintegration of unitary representations.
 
This article is organized as follows.
Section \ref{Sec-not} is devoted to definitions and basic properties of unitary representations that will be used in the rest of the paper. In Section 
\ref{sec-crossed}
 we define the crossed product $*$-algebra $\cA^\circ$. 
In Section \ref{seccomp}
we construct the $C^*$-algebra $\cA:=\cA(G,\g g)$ as the 
completion of the crossed product algebra $\cA^\circ$  with respect to a suitable 
$C^*$-seminorm.
In Section \ref{sec-mult} we prove that under the $G$-action on $\cA$ by multipliers, 
orbit maps of elements of $\cA^\circ$ are smooth. 
In Section \ref{sec-nondeg} we describe the canonical bijective correspondence between unitary representations of the Lie supergroup $(G,\g g)$ and  the (ungraded) nondegenerate $*$-representations of 
$\cA$. Finally, in Section 
\ref{sec-unique} we obtain our liminality results for  $C^*$-algebras of a broad class of Lie supergroups, including the nilpotent and classical simple ones.\\[2mm]
\textbf{Acknowledgement.} During the completion of this project, the second author was supported by an NSERC Discovery Grant and  the Emerging Field Project
``Quantum Geometry'' of FAU Erlangen--N\"{u}rnberg.

\section{Basic definitions}
\label{Sec-not}

We begin by a rapid review of Lie supergroups (from the Harish--Chandra pair viewpoint) and their unitary representations. For a more elaborate reference, see \cite{varadarajan}.

Throughout this paper, $\Z/2\Z:=\{\oline 0,\oline 1\}$ denotes the field with two elements. If $V=V_\eev\oplus V_\ood$ is a $\Z/2\Z$-graded vector space, then the parity of a homogeneous element $x\in V$ is denoted by $|x|\in\Z/2\Z$.
\begin{dfn}
\label{def-blsupergroup}
A \emph{Lie supergroup} is an ordered pair $(G,\g g)$ with the following properties.
\begin{itemize}
\item[(i)] $G$ is a Lie group. 
\item[(ii)] $\g g=\g g_\eev\oplus\g g_\ood$ is a  Lie superalgebra over $\R$.
\item[(iii)] $\g g_\eev$ is the Lie algebra of $G$. 
\item[(iv)] There exists a group homomorphism 
$\Ad:G\to\mathrm{Aut}(\g g)$, defining a smooth action 
$G\times \g g\to \g g$, such that 
\[
\Ad(g)x=\dd c_g(\yek)(x)
\text{ and }
\dd\Ad^y(\yek)(x)=[x,y]
\]
for every $x\in\g g _\eev$, $y\in \g g$, and $g\in G$,
where
 $c_g:G\to G$ is defined by $c_g(g'):=gg'g^{-1}$ and $\Ad^y:G\to\g g$ is defined by $\Ad^y(g):=\Ad(g)y$.
\end{itemize}
\end{dfn}
In this article we assume that $\dim\g g<\infty$.
The Lie supergroup $(G,\g g)$ is called \emph{connected} if $G$ is a connected Lie group.
\begin{rmk}
Here we should clarify that the condition given in Definition \ref{def-blsupergroup}(iv)
is identical to the ones given in our previous papers 
\cite[Def.\ 4.6.3(iv)]{nsTransf} and \cite[Def.\ 7.1(iv)]{nsMZ}. More precisely, in \cite{nsTransf} and \cite{nsMZ} we tacitly assume that $\Ad$ is an extension of the adjoint action of $G$ on $\g g_\eev$. \end{rmk}

Let $(\pi,\ccH)$ be a unitary representation of a Lie group $G$. For  $x\in\Lie(G)$ and $v\in\ccH$, we set 
\[
{\dd\pi}(x)v:=\lim_{t\to 0}\frac{1}{t}\left(\pi(e^{tx})v-v\right),
\]
whenever the limit exists. Here $e^{tx}:=\exp(tx)$ denotes the exponential map of $G$.

\begin{dfn}
\label{defnungradgrad} 
Let $(G,\g g)$ be a Lie supergroup. 
A \emph{unitary representation} of 
$(G,\g g)$ is a triple
$
(\pi,\rho^\pi,\ccH)$
which satisfies
the following properties.
\begin{itemize} 
\item[(i)] 
$\ccH$ has a $\Z/2\Z$-grading, that is, 
$\ccH=\ccH_\eev\oplus\ccH_\ood$, and
$(\pi,\ccH)$ is a smooth unitary representation of $G$, such that $\pi(g)$ preserves the $\Z/2\Z$-grading of $\ccH$ for every $g\in G$.
\item[(ii)] $\rho^\pi:\g g\to \End_\mathbb C(\ccH^\infty)$  
 is a representation of the Lie superalgebra $\g g$, where
 $\ccH^\infty_{}=\ccH^\infty_\eev\oplus\ccH^\infty_\ood$ is the subspace consisting of all $v\in\ccH$ 
 for which the orbit map $G\to \ccH$, $g\mapsto \pi(g)v$ is smooth.
\item[(iii)]   
$\rho^\pi(x)={\dd\pi}(x)\big|_{\ccH^\infty}$ for every $x\in\g g_\eev$.
\item[(iv)] For every $x\in \g g_\ood$, the
operator $e^{-\frac{\pi i}{4}}\rho^\pi(x)$ is  
symmetric. That is,
\[
-i\rho^\pi(x) \subseteq \rho^\pi(x)^*.
\] 
\item[(v)] 
$\pi(g)\rho^\pi(x)\pi(g)^{-1}=\rho^\pi(\Ad(g)x)$ for every $g\in G$ and every $x\in\g g_\ood$.
\end{itemize}
\end{dfn}

\begin{rmk}
\label{rem-weaker(e)}
By \cite[Prop. 6.13]{nsMZ},  
the condition given in Definition 
\ref{defnungradgrad}(v) 
follows from the weaker condition that for every element of the component group $G/G^\circ$, there exists a coset representative $g\in G$ such that
\[
\pi(g)\rho^\pi(x)\pi(g^{-1})=\rho^\pi(\Ad(g)x)
\text{ for every }x\in\g g.
\] 
\end{rmk}

\begin{rmk}
\label{prp-dirsepp}
As in \cite[Def.\ 6.7.1]{nsTransf}, a unitary representation $(\pi,\rho^\pi,\ccH)$
is called \emph{cyclic} if there exists a 
vector $v\in\ccH^\infty_\eev$ such that
$\pi(G)\rho^\pi(\bfU(\g g_\C))v$ spans a dense subspace
of $\ccH$, where 
$\g g_\C:=\g g\otimes_\R^{}\C$ and
$\bfU(\g g_\C)$ denotes the universal enveloping algebra of $\g g_\C$. A standard Zorn Lemma argument shows that  every unitary representation can be written as a direct sum of representations which are cyclic up to parity change. Furthermore, in \cite[Thm 6.7.5]{nsTransf} a GNS construction is given which results in a correspondence between  cyclic unitary representations
and positive definite superfunctions on $(G,\g g)$.
\end{rmk}

Let $(\pi,\rho^\pi,\ccH)$ be a unitary representation of $(G,\g g)$.
We equip the space $\ccH^\infty$ 
with the topology induced by 
the seminorms $v\mapsto \|{\dd\pi}(D)v\|$, for all $D\in\bfU(\g g_\eev)$.  This topology makes $\ccH^\infty$ a Fr\'{e}chet space.
\begin{prp}
\label{rhopicont}
For every $x\in\g g$, the map $\rho^\pi(x):\ccH^\infty\to\ccH^\infty$
is continuous. 
\end{prp}

\begin{proof}
Continuity for $x\in\g g_\eev$ is standard, and therefore we will assume that $x\in\g g_\ood$. We need to prove that for every $x_1,\ldots x_\ell\in\g g_\eev$, for $D:=x_1\cdots x_\ell\in\bfU(\g g_\eev)$, the map 
\begin{equation}
\label{eq-Hinfpitov}
\ccH^\infty\to\R\ ,\ v\mapsto \|\dd\pi(D)\rho^\pi(x)v\|
\end{equation}
is continuous at $0\in\ccH^\infty$. 
First assume that $\ell=0$, so that $D=1\in\bfU(\g g_\eev)$. In this case, continuity of \eqref{eq-Hinfpitov} follows from the inequality
\begin{align*}
\|\rho^\pi(x)v\|^2&=\left|\lag v,\rho^\pi(x)^2v\rag\right|
\leq\frac12\|v\|\cdot\|\dd\pi([x,x])v\|
\end{align*}
and the definition of the topology of $\ccH^\infty$.
To prove continuity of \eqref{eq-Hinfpitov} for $\ell\geq 1$, we  use the relation
\[
x_1\cdots x_\ell x=x x_1\cdots x_\ell+\sum_{i=1}^\ell x_1\cdots x_{i-1}[x_i,x]x_{i+1}\cdots x_\ell
\]
and induction on $\ell$.
\end{proof}
\begin{dfn}
A \emph{multiplier} of an associative algebra $\cA$  is a pair 
$(\lambda,\rho)$ of linear maps $\cA\to\cA$ 
which satisfy
the relations
\[
\lambda(ab)=\lambda(a)b,\
\rho(ab)=a\rho(b),\
\text{and }a\lambda(b)=\rho(a)b
\]
for every $a,b\in\cA$.
\end{dfn}
If $\cA$ is a $*$-algebra, then the multipliers of $\cA$ form a $*$-algebra, denoted by 
$M(\cA)$, with multiplication
and involution defined by
\begin{equation}
\label{mult***}
(\lambda,\rho)(\lambda',\rho'):=(\lambda\lambda',\rho'\rho)
\text{ and }
(\lambda,\rho)^*:=(\rho^*,\lambda^*),
\end{equation}
where $\lambda^*(a):=\lambda(a^*)^*$ and $\rho^*(a)=\rho(a^*)^*$.

\section{The crossed product  $*$-algebra $\cA^\circ$
}
\label{sec-crossed}
Fix a Lie supergroup $(G,\g g)$. Set  $G_\boldeps:=G\times\{\yek ,\boldeps\}$
such that $\boldeps^2=\yek$, and define $\Ad(\boldeps)x:=(-1)^{|x|}x$ for every homogeneous $x\in\g g$. We endow $G_\boldeps$ with the product topology. 
Clearly $(G_\boldeps,\g g)$ is also a Lie supergroup. 
Let $\cD(G_\boldeps)$ be the convolution algebra of test functions (i.e., smooth compactly supported complex-valued functions) on $G_\boldeps$. The convolution on $\cD(G_\boldeps)$ is defined by
\[
(f_1\star f_2)(g'):=\int_{G_\boldeps} f_1(g)f_2(g^{-1}g')dg,
\]
where $dg$ is the left-invariant Haar measure. 
The $*$-algebra structure 
is given by the involution 
\[ \breve f(g):=\Delta(g)^{-1}\oline{f(g^{-1})},\]
where
$g\mapsto \Delta(g)$ is the modular function satisfying $d(gg')=\Delta(g')dg$. 
From now on, we set
\[
\LL_gf(g'):=f(g^{-1}g')
\text{
and\ \
}
\sfR_x f(g)
:=
\lim_{t\to 0}
\frac1t\left(
\LL_{e^{tx}}^{}f(g)-f(g)\right),
\]
for $x\in\g g_\eev$, $g,g'\in G_\boldeps$, $f\in\cD(G_\boldeps)$, and $t\in\R$.

Set $\g g_\C:=\g g\otimes_\R^{}\C$.
For every $g\in G_\boldeps$, let 
$\alpha_g:\bfU(\g g_\C)\to\bfU(\g g_\C)$ denote the automorphism that is 
canonically induced by $\Ad(g):\g g\to\g g$.
Our next goal is to define a \emph{crossed product} 
$*$-algebra
$\cA^\circ=\cA^\circ(G,\g g)$.
As a complex vector space, 
\[
\cA^\circ:=\bfU(\g g_\C)\otimes \cD(G_\boldeps)
.
\] We identify $\cA^\circ$ with a subspace of the vector space of $\bfU(\g g_\C)$--valued functions on $G_\boldeps$ in the canonical way. Using this identification, we define a  multiplication 
and a complex conjugation on $\cA^\circ$ by the relations
\begin{equation}
\label{multofA}
(D_1\otimes f_1)(D_2\otimes f_2)(g'):=
\int_{G_\boldeps}
f_1(g)f_2(g^{-1}g')D_1\alpha_g( D_2)dg
\end{equation}
and 
\begin{equation}
\label{antlininv}
(D\otimes f)^*(g):=
\Delta(g^{-1})
\oline{f(g^{-1})}\alpha_g(D^\dagger),
\end{equation}
where  the map 
$
x\mapsto x^\dagger
$ 
is the anti-linear anti-automorphism of $\bfU(\g g_\C)$ uniquely defined by
\begin{equation}
\label{daggerrr}
x^\dagger:=\begin{cases}
-x& \text{ if }x\in\g g_\eev,\\
e^{-\frac{\pi i}{2}}x&\text{ if }x\in\g g_\ood.
\end{cases}
\end{equation}
In particular,
\[
(D_1\otimes f_1)(1\otimes f_2)=D_1\otimes (f_1\star f_2).
\]
Every $g\in G_\boldeps$ yields a multiplier $(\lambda_g,\rho_g)$ 
of $\cA^\circ$ by setting
\begin{equation}
\label{eq-lamgrhog}
\lambda_g(D\otimes f)
:=
\alpha_g(D)\otimes \LL_gf 
\text{ and }
\rho_g(D\otimes f):=
D\otimes\Delta(g^{-1})\mathrm{R}_{g^{-1}}f
,
\end{equation}
where $\mathrm{R}_gf(g'):=f(g'g)$.

The algebra $\cA^\circ$ is not necessarily unital. Nevertheless, we have the following lemma.
\begin{lem}
\label{DMAA=A}
Every $a\in\cA^\circ$ can be written as a finite sum \[
a=a_1b_1+\cdots +a_mb_m,
\] where $a_1,\ldots,a_m,b_1,\ldots,b_m\in\cA^\circ$. 
In other words, $\cA^\circ=\cA^\circ\cA^\circ$.
\end{lem}
\begin{proof}
It is enough to prove the statement for elements of $\cA^\circ$ of the form $D\otimes f$. By the Dixmier--Malliavin Theorem \cite{dixmal}, we can write 
\[
f=f_1\star h_1+\cdots+f_m\star h_m,
\] where
$f_1,\ldots,f_m,h_1,\ldots,h_m\in\cD(G_\boldeps)$. 
It follows that
\[
D\otimes f=(D\otimes f_1)(1\otimes h_1)+\cdots+
(D\otimes f_m)(1\otimes h_m).\qedhere
\]
\end{proof}

Every unitary representation
$(\pi,\rho^\pi,\ccH)$ of $(G,\g g)$
extends to a unitary representation of $(G_\boldeps,\g g)$ by setting $\pi(\boldeps)v=(-1)^{|v|}v$ for every homogeneous $v\in\ccH$. From now on, 
we assume that every unitary representation of $(G,\g g)$ has been  extended to $(G_\boldeps,\g g)$ as indicated above.

Fix a unitary representation $(\pi,\rho^\pi,\ccH)$ of $(G,\g g)$. Let $D\otimes f\in\cA^\circ$, and as usual set \[
\pi(f):=\int_{G_\boldeps}f(g)\pi(g)dg.
\]   Note that 
$\|\pi(f)\|\leq \|f\|_{L^1}$.
By G\aa rding's Theorem we know that 
$\pi(f)\ccH\subseteq\ccH^\infty$, so that the
linear map 
\[
\rho^\pi(D)\pi(f):\ccH\to \ccH
\] is well-defined. 
\begin{prp}
\label{fDincAbdd}
Let $D\otimes f\in\cA^\circ$. There exists a  constant $M_{D\otimes f}^{}>0$  such that 
\[
\|\rho^{\pi}(D)\pi(f)
\|\leq M_{D\otimes f}^{}
\] for every 
unitary representation $(\pi,\rho^{\pi},\ccH)$ 
of $(G,\g g)$.
\end{prp}
\begin{proof}
Note that
$\pi(g)\pi(f)=\pi(\LL_gf)$ for $g\in G_\boldeps$ and $f\in\cD(G_\boldeps)$, 
and  for $x\in\g g_\eev$ we have
$
\lim_{t\to 0}
\left\|
\frac1t\left(
\LL_{e^{tx}}f-f
\right)-\sfR_xf\right\|_{L^1}=0
$.
Thus  for every $v\in\ccH$, we obtain that
\begin{align}
\label{ddpesss}
\dd\pi(x)\pi(f)v
&=
\lim_{t\to 0}
\frac1t
\left(\pi(e^{tx}_{})
\pi(f)v-\pi(f)v\right)\\
&=
\lim_{t\to 0}
\frac1t
\left(\pi(\LL_{e^{tx}}f)v-\pi(f)v\right)
=\pi(\sfR_xf)v.
\notag
\end{align}
By induction, from \eqref{ddpesss} it follows that 
\begin{align}
\label{ddpiepsDrho}
\dd\pi(D)\pi(f)v=\pi(\sfR_Df)v
\ \text{ for }D\in\bfU(\g g_\eev),
\ f\in\cD(G_\boldeps),
\text{ and }v\in\ccH.
\end{align}
If $x\in \g g_\ood$, then from \eqref{ddpesss}
it follows that
\begin{align}
\label{xinggg1}
\|\rho^\pi(x)\pi(f)v\|^2
&=
\lag \rho^\pi(x)\pi(f)v,\rho^\pi(x)\pi(f)v\rag=
\frac12|\lag \rho^\pi([x,x])\pi(f)v,\pi(f)v\rag|\\
&\leq \frac12
\|\rho^\pi([x,x])\pi(f)v\|
\cdot
\|\pi(f)v\|\leq
\frac12
\|\sfR_{[x,x]}f\|_{L^1}^{}\cdot
\|f\|_{L^1}^{}\cdot \|v\|^2.
\notag
\end{align}
Similarly, if $x_1,\ldots, x_d\in\g g_\ood$ for some $d>1$, then
\begin{align}
\label{rhoPIcdotsxd}
\notag
\|\rho^\pi(&x_1)\cdots\rho^\pi(x_d)\pi(f)v\|^2
=
\lag 
\,\rho^\pi(x_1)\cdots\rho^\pi(x_d)\pi(f)v
,
\rho^\pi(x_1)\cdots\rho^\pi(x_d)\pi(f)v
\,\rag\\
&\leq
\frac12\|\rho^\pi(x_2)\cdots\rho^\pi(x_d)\pi(f)v\|
\cdot
\|\rho^\pi([x_1,x_1])
\rho^\pi(x_2)
\cdots\rho^\pi(x_d)\pi(f)v\|.
\end{align}
Furthermore,
\begin{align}
\label{rhopij=2d}
&\rho^\pi([x_1,x_1])\rho^\pi(x_2)
\cdots\rho^\pi(x_d)\pi(f)v\\
\notag
&=
\sum_{j=2}^d
\rho^\pi(x_2)\cdots\rho^\pi([x_1,x_1],x_j])
\cdots\rho^\pi(x_d)\pi(f)v
+
\rho^\pi(x_2)\cdots\rho^\pi(x_d)
\pi(\sfR_{[x_1,x_1]}f)v.
\end{align}
By the PBW Theorem, it is enough  to prove the 
statement of the proposition when 
\[
D=
y_1\cdots y_{\ell'}
x_1\cdots x_\ell
,
\]
where $x_1,\ldots,x_\ell\in\g g_\eev$ and $y_1,\ldots,y_{\ell'}\in\g g_\ood$.
From \eqref{ddpiepsDrho}, 
\eqref{xinggg1},
\eqref{rhoPIcdotsxd}, and \eqref{rhopij=2d}, and by induction on $\ell'$, it follows that $\|\rho^{\pi}(D)\pi(f)\|$ is bounded above by a constant which is expressible 
in terms of the $L^1$-norms of derivatives of $f$.
\end{proof}

\section{The $C^*$-algebra $\cA:=\cA(G,\g g)$}
\label{seccomp}

For a unitary representation $(\pi,\rho^\pi,\ccH)$ of $(G,\g g)$, we define the linear map
$\whpi:\cA^\circ\to B(\ccH)$  by setting
\begin{equation}
\label{whpidef}
\whpi(D\otimes f):=\rho^{\pi}(D)\pi(f)
\text{ for every }D\otimes f\in\cA^\circ,
\end{equation}
and then extending $\whpi$ to a linear map on $\cA^\circ$. 
Consider the seminorm on $\cA^\circ$ defined by
\begin{equation}
\label{seminorm}
\|a\|:=\sup_{(\pi,\rho^\pi,\ccH)}\|\whpi(a)\|
\end{equation}
where the supremum is taken over all unitary equivalence classes of cyclic unitary representations ${(\pi,\rho^\pi,\ccH)}$ of $(G,\g g)$. From
Proposition \ref{fDincAbdd} it follows that $\|a\|<\infty$.

\begin{lem}
\label{lem-pabpastar}
$\whpi$ is a $*$-representation of $\cA^\circ$.
\end{lem}
\begin{proof}
First we prove that
$\whpi(ab)=\whpi(a)\whpi(b)$ for every $a,b\in\cA^\circ$. It is enough to assume that $a=D_1\otimes f_1$ and $b=D_2\otimes f_2$. Choose $\eta_1,\ldots,\eta_r\in \cC^\infty(G_\boldeps)$ 
and $E_1,\ldots,E_r\in\bfU(\g g_\C)$  such that
$\alpha_g(D_2)=\sum_{i=1}^r\eta_i(g)E_i$.
Fix $v\in\ccH$ and set $w:=\pi(f_2)v$. Then $w\in\ccH^\infty$ and therefore the map $G_\boldeps\to \ccH^\infty$, $g\mapsto \pi(g)w$ is smooth \cite[Prop. 2.1]{Poulsen}. Using
Proposition \ref{rhopicont} we obtain   
\begin{align*}
\whpi(a)\whpi(b)v&=\rho^\pi(D_1)\int_{G_\boldeps}
f_1(g)\pi(g)\rho^\pi(D_2)\pi(f_2)v dg\\
&=
\rho^\pi(D_1)\int_{G_\boldeps}
f_1(g)\rho^\pi(\alpha_g(D_2))\pi(g)\pi(f_2)v dg\\
&=\rho^\pi(D_1)\sum_{i=1}^r\rho^\pi(E_i)\pi(\eta_if_1)
\pi(f_2)v\\
&
=\whpi\left(\sum_{i=1}^rD_1E_i\otimes (\eta_if_1 * f_2)\right)v=\whpi(ab)v.
\end{align*}

The equality $\whpi(a)^*=\whpi(a^*)$ can be verified 
by a similar calculation, using the relation
\[
\lag\rho^\pi(D\otimes f)v,w\rag
=
\lag\rho^\pi(D)\pi(f)v,w\rag=\lag v,\pi(\breve f)\rho^\pi(D^\dagger)w\rag
\quad \mbox{ for } \quad v,w \in \ccH^\infty,
\] where $\breve f(g)=\Delta(g)^{-1}\oline{f(g^{-1})}$.
\end{proof}
We are now ready to define $\cA:=\cA(G,\g g)$.
From Lemma \ref{lem-pabpastar} it follows that the map $a\mapsto a^*$ is an isometry of $\cA^\circ$ and $\|aa^*\|=\|a\|^2$.
Set $\cA^\circ_-:=\{a\in\cA^\circ\,:\,\|a\|=0\}$ and let 
$\cA$ denote the completion of the quotient $\cA^\circ/\cA^\circ_-$ with respect to its induced norm. It is straightforward to check that $\cA$ is  a $C^*$-algebra. 

\begin{lem}
\label{lem-gammafDh}
Let $f\in \cD(G_\boldeps)$ and $D\otimes h\in \cA^\circ$. Then the map 
\[
\gamma_{f,D,h}^{}:G_\boldeps\to\cA\ ,\ 
g\mapsto f(g)\alpha_g(D)\otimes \LL_gh
\]
is continuous and 
\begin{equation}
\label{gfDh}
\int_{G_\boldeps}\gamma_{f,D,h}^{}(g)dg=(1\otimes f)(D\otimes h).
\end{equation}
\end{lem}
\begin{proof}
Choose $E_1,\ldots, E_r\in\bfU(\g g_\C)$ and $\eta_1,\ldots,\eta_r\in \cC^\infty(G_\boldeps)$ such that 
$\alpha_g(D)=\sum_{i=1}^r\eta_i(g)E_i$ for every $g\in G_\boldeps$. It follows that
\[
\gamma_{f,D,h}^{}(g)=\sum_{i=1}^rE_i\otimes f(g)\eta_i(g)\LL_gh.
\]
Next we prove that, for every $1\leq i\leq r$, the map $G_\boldeps\to\cA$, $g\mapsto E_i\otimes\LL_gh$ is continuous. Since we can replace $h$ by $\LL_gh$, 
it suffices to prove continuity at $\yek\in G_\boldeps$.
To this end, we need to show  that \[
\lim_{g\to \yek}\left(\sup_{(\pi,\rho^\pi,\ccH)}
\big\|\rho^\pi_{}(E_i)\pi(\LL_gh-h)\big\|\right)=0.
\] By an argument similar to the proof of Proposition \ref{fDincAbdd}, the latter statement can be reduced to showing that for every $D\in\bfU(\g g_\eev)$,
\[
\lim_{g\to \yek}\big\|
\sfR_{D}^{}(\LL_gh-h)
\big\|_{L^1}^{}=0.
\] 
This is straightforward.

Next we prove \eqref{gfDh}. From 
\eqref{multofA}
 it follows that
\begin{align}
\label{1otfDoth}
(1\otimes f)(D\otimes h)(g')&=
\sum_{i=1}^r
\int_{G_\boldeps}f(g)(\LL_gh)(g')\eta_i(g)E_idg
=\sum_{i=1}^r (f\eta_i\star h)(g')E_i.
\end{align}
The left-regular representation of $G_\boldeps$ on $L^1(G_\boldeps)$ is strongly continuous, and its integrated representation is given by convolution, that is, 
$\int_{G_\boldeps}f(g)\LL_gh\ dg=f\star h$ for every $f,h\in L^1(G_\boldeps)$.
It follows that
 \begin{align}
 \label{gafDhg}
\int_{G_\boldeps}\gamma_{f,D,h}^{}(g)dg&=
\sum_{i=1}^r E_i\otimes \int_{G_\boldeps}f(g)\eta_i(g)\LL_ghdg
=\sum_{i=1}^r E_i\otimes \big((f\eta_i)\star h\big).
\end{align} 
Equality \eqref{gfDh} now follows
from \eqref{1otfDoth} and \eqref{gafDhg}. 
\end{proof}

\section{Multipliers of $\cA$ and $\cA^\circ$}
\label{sec-mult}

For every $g\in G_\boldeps$, let $(\lambda_g,\rho_g)$ 
be the multiplier of $\cA^\circ$ that is defined  in \eqref{eq-lamgrhog}.
It is straightforward to verify that $\lambda_g$ and $\rho_g$ are isometries of $\cA^\circ$, and therefore the multiplier $(\lambda_g,\rho_g)$ extends uniquely to a multiplier of $\cA$.
For every $g\in G_\boldeps$, the  map
\begin{equation}
\label{eqEtaGga}
\eta_G(g):\cA\to\cA\ ,\ a\mapsto\lambda_g(a)
\end{equation}
 is an isometry and $\eta_G(gg')=\eta_G(g)\eta_G(g')$.
\begin{prp}
\label{cAsmooth}
For every $a\in\cA^\circ$, the map
\[
G\to\cA\ ,\ 
g\mapsto \eta_G(g)a
\] is smooth.
\end{prp}

\begin{proof}
It suffices to prove that the orbit map of every $D\otimes f\in\cA^\circ$ is smooth.
Set
\[
\whpi_u:=\bigoplus_{(\pi,\rho^\pi,\ccH)}\!\!\whpi
\ \text{ and }\ 
(\pi_u,\ccH_u):=
\bigoplus_{(\pi,\rho^\pi,\ccH)}\!\!(\pi,\ccH)
\]
where the direct sums are over unitary equivalence classes of cyclic unitary representations of $(G,\g g)$. 
Then $(\pi_u,\ccH_u)$ is a smooth unitary representation of 
$G_\boldeps$, and 
the map
$\whpi_u:\cA^\circ\to B(\ccH_u)$ is an isometry. Furthermore,
\begin{align*}
\whpi_u(\eta(g)a)
=\pi_u(g)\whpi_u(a)\text{ for }g\in G_\boldeps\text{ and }
a\in\cA^\circ.
\end{align*}
Consequently, to complete the proof it suffices to show that for every $a\in\cA^\circ$, the map
\[
G\to B(\ccH_u)\ ,\ 
g\mapsto \pi_u(g)\whpi_u(a).
\]
is smooth. The latter statement follows from
\cite[Theorem 2.11]{nsz}.
\end{proof}

By \cite[Props. VIII.1.11/18]{doranfell}, every multiplier of $\cA$ is bounded and 
the multipliers of $\cA$ form a unital Banach $*$-algebra $M(\cA)$ with multiplication and complex conjugation \eqref{mult***} and the norm defined by $\|(\lambda,\rho)\|:=\max\{\|\lambda\|,\|\rho\|\}$.
Furthermore, the multipliers $(\lambda_g,\rho_g)$ for $g\in G_\boldeps$ are unitary, that is, 
\[
(\lambda_g,\rho_g)(\lambda_g,\rho_g)^*=1\in M(\cA).
\]

\section{Nondegenerate $*$--representations of $\cA$}
\label{sec-nondeg}
In this section we prove that the category of unitary representations of $(G,\g g)$ is isomorphic to the category of nondegenerate (in the sense of 
\cite[Def.\ V.1.7]{doranfell})
$*$--representations of the $C^*$-algebra $\cA=\cA(G,\g g)$. 

\begin{prp}
\label{prop-pirhotowhpi}
Let $(\pi,\rho^\pi,\ccH)$ 
be a unitary representation
of a Lie supergroup $(G,\g g)$.
Then the  $*$--representation $\whpi$ defined in {\rm Lemma~\ref{lem-pabpastar}}
 extends in a unique way to a nondegenerate $*$-representation 
\[
\whpi:\cA\to B(\ccH).
\]
\end{prp}
\begin{proof}


From \eqref{seminorm}
and  Remark \ref{prp-dirsepp}
it follows that 
$\|\whpi(a)\|\leq \|a\|$ for every $a\in\cA^\circ$. The existence and uniqueness of the extension 
$\whpi:\cA\to B(\ccH)$
now follows immediately. Nondegeneracy of 
$\whpi$ follows from the equality $\whpi(1\otimes f)=\pi(f)$ for $f\in\cD(G_\boldeps)$.
\end{proof}
We now give a construction of a unitary representation
of $(G,\g g)$ from a nondegenerate $*$-representation 
$
\widehat\pi:\cA\to B(\ccH)
$ of  $\cA$.
By \cite[Prop. VIII.1.11]{doranfell} and
\cite[Prop. VIII.1.12]{doranfell}, there exists a unique extension of $\whpi$ to a 
$*$-representation 
$\whpi:M(\cA)\to B(\ccH)$ of the multiplier algebra $M(\cA)$ satisfying
\begin{equation}
\label{whpilra=la}
\whpi\big((\lambda,\rho)\big)\whpi(a)=\whpi(\lambda(a))\ \text{ for }\ (\lambda,\rho)\in M(\cA)\text{ and } a\in\cA.
\end{equation}
Set
\begin{equation}
\label{eqpig=hatttt}
\pi(g):=\whpi\big((\lambda_g,\rho_g)\big)\text{ for every }g\in G_\boldeps.
\end{equation}
From \eqref{whpilra=la} it follows that the subspace $\ccH^\circ:=\whpi(\cA^\circ)\ccH$ is invariant under $\pi(g)$ for every $g\in G_\boldeps$. 
\begin{lem}
\label{piissmoothrp}
For every $v\in\ccH$,  
the map $G_\boldeps\to\ccH$, $g\mapsto \pi(g)v$ is smooth if and only if $v\in\ccH^\circ$. In particular, 
$(\pi,\ccH)$ is a smooth unitary representation of $G$.
\end{lem}
\begin{proof}
First we show that for every $v\in \ccH^\circ$, the orbit map
$G_\boldeps\to \ccH$, $g\mapsto \pi(g)v$ is smooth. 
Assume that $v=\whpi(a)w$ for $a\in\cA^\circ$ and $w\in\ccH$. Then  
\[
\pi(g)v=\whpi\big(\eta_G(g)a\big)w
,
\]
where $\eta_G(g):\cA\to\cA$ is defined in 
\eqref{eqEtaGga}. Since the map $\cA\to\ccH$, $a\mapsto \whpi(a)w$ is continuous and linear, Proposition 
\ref{cAsmooth}
 implies that the orbit map $g\mapsto \pi(g)v$ is smooth. 
 
Next we observe that $\ccH^\circ$ is a dense subspace of 
 $\ccH$, because  $\cA^\circ$ is a dense subspace of 
 $\cA$. Therefore the representation $(\pi,\ccH)$ is smooth.
 
Finally, we prove that every smooth vector of $(\pi,\ccH)$ belongs to $\ccH^\circ$. By the Dixmier--Malliavin Theorem, it is enough to show that 
\begin{equation}
\label{pifpihf}
\pi(f)=\whpi(1\otimes f)\text{ for every } f\in\cD(G_\boldeps),
\end{equation} 
where 
$
\pi(f)v:=\int_{G_\boldeps}f(g)\pi(g)v\, dg\ \text{ for }v\in\ccH
$. Since both sides of \eqref{pifpihf} are bounded operators and $\ccH^\circ$ is dense in $\ccH$, it is enough to prove that \[
\pi(f)\whpi(D\otimes h)v=\whpi\big((1\otimes f)(D\otimes h )\big)v\text{ for }D\otimes h\in\cA^\circ\text{ and }v\in\ccH.
\] 
Let  $\gamma_{f,D,h}^{}$ be defined as in
Lemma \ref{lem-gammafDh}. 
From \eqref{eqpig=hatttt} and \eqref{whpilra=la} it follows that
for every $v\in\ccH$,
\begin{align*}
\pi(f)\whpi(D\otimes h)v&=\int_{G_\boldeps}
f(g)\pi(g)\whpi(D\otimes h)v\, dg
\\
&=
\int_{G_\boldeps}\whpi\big(\gamma_{f,D,h}^{}(g)\big)v\, dg=\whpi\left(
\int_{G_\boldeps}\gamma_{f,D,h}^{}(g)dg\right)\!\!v,
\end{align*}
and from 
\eqref{gfDh}
it follows that 
$\pi(f)\whpi(D\otimes h)v=
\whpi\big((1\otimes f)(D\otimes h)\big)v$.
\end{proof}
Set \[
\pi^\circ(a):=\whpi(a)\big|_{\ccH^\circ}
\text{ for 
every }a\in\cA^\circ.
\]
From Lemma \ref{DMAA=A} and Lemma \ref{piissmoothrp} it follows that  $(\pi^\circ,\ccH^\circ)$ is a non-degenerate $*$-representation 
of $\cA^\circ$
in the sense defined in \cite[Def.\ IV.3.17]{doranfell}. Therefore by \cite[Prop. VIII.1.9]{doranfell} there exists a unique extension of $\pi^\circ$ to a $*$-representation 
$
\pi^\circ: M(\cA^\circ)\to\End_\C(\ccH^\circ)
$
satisfying 
\[
\pi^\circ\big((\lambda,\rho)\big)\pi^\circ(a)=\pi^\circ(\lambda(a))
\,\text{ for }\,(\lambda,\rho)\in M(\cA^\circ)
\text{ and }a\in\cA^\circ.
\] From the latter equality, Lemma \ref{DMAA=A}, and \eqref{whpilra=la}, it follows that 
$
\pi^\circ\big((\lambda_g,\rho_g)\big)=\pi(g)\big|_{\ccH^\circ}
$
for every $g\in G_\boldeps$.

For every $x\in\g g$, 
let
$(\lambda_x,\rho_x)\in M(\cA^\circ)$
be the multiplier  defined by
\[
\lambda_x(D\otimes f):=xD\otimes f\ \text{ and }\ 
\rho_x(D\otimes f)(g):=f(g)D\alpha_g(x)\,\text{ for }g\in G_\boldeps.
\]
It is straightforward to verify that $(\lambda_x,\rho_x)^*=(\lambda_{x^\dagger},\rho_{x^\dagger})$ for every $x\in\g g$, where $x^\dagger$ is defined as in 
\eqref{daggerrr}.
For every $x\in\g g$, 
we define a linear map
\begin{equation}
\label{rhotopix}
\rho^\pi(x):\ccH^\circ\to\ccH^\circ
\ ,\ 
v\mapsto \pi^\circ\big((\lambda_x,\rho_x)\big)v.
\end{equation}
Since 
$\pi(\boldeps)^2=\yek$, we obtain a $\Z/2\Z$-grading 
$\ccH=\ccH_\eev\oplus\ccH_\ood$ by the $\pm 1$ eigenspaces of $\pi(\boldeps)$, 
i.e.\ 
\[
\ccH_\eev:=\{v\in\ccH\,:\,\pi(\boldeps)v=v\}\text{ and }
\ccH_\ood:=\{v\in\ccH\,:\,\pi(\boldeps)v=-v\}.
\]
Since $\pi(\boldeps)$ leaves $\ccH^\circ$ invariant, the $\Z/2\Z$-grading of $\ccH$ induces a $\Z/2\Z$-grading  $\ccH^\circ=\ccH_\eev^\circ\oplus\ccH_\ood^\circ$ on $\ccH^\circ$. We now prove the following proposition.
\begin{prp}
\label{prp-isunGGg}
$(\pi,\rho^\pi,\ccH)$ is a unitary representation of $(G,\g g)$. 
\end{prp}

\begin{proof}
Every $g\in G$ commutes with $\boldeps$, and therefore 
$\pi(g)$ preserves the $\Z/2\Z$-grading of $\ccH$.
If $x\in\g g_\ood$, then $(\lambda_x,\rho_x)^*=(\rho_{-ix},\lambda_{-ix})$ in 
$M(\cA^\circ)$, and it follows that the operator $e^{-\frac{\pi i}{4}}\rho^\pi(x)$ is symmetric.
For every $x\in\g g$ and $g\in G_\boldeps$, we have
\[
(\lambda_g,\rho_g)(\lambda_x,\rho_x)(\lambda_{g^{-1}},\rho_{g^{-1}})=
(\lambda_{\Ad(g)x},\rho_{\Ad(g)x}),
\]
and consequently,
\begin{equation}
\label{adjAD}
\pi(g)\rho^\pi(x)\pi(g)^{-1}=\rho^\pi(\Ad(g)x).
\end{equation}
In particular, 
from \eqref{adjAD} for
$g=\boldeps$ it follows that  $\rho^\pi(x)\in\End_\C(\ccH^\circ)_\eev$ 
for $x\in\g g_\eev$ and 
$\rho^\pi(x)\in\End_\C(\ccH^\circ)_\ood$
for $x\in\g g_\ood$. 
The relation
$
\rho^\pi([x,y])=[\rho^\pi(x),\rho^\pi(y)]
$
for $x,y\in\g g$
follows from the corresponding relation in the multiplier algebra $M(\cA^\circ)$.
Finally, we prove that $\rho^\pi(x)={\dd\pi}(x)\big|_{\ccH^\circ}$ for every $x\in\g g_\eev$. Fix $a\in\cA^\circ$ and $v\in\ccH^\circ$, and set 
\[
\phi_{a,t}:=\frac{1}{t}
\big(\pi(e^{tx})\pi^\circ(a)v-
\pi^\circ(a)v\big)-\rho^\pi(x)\pi^\circ(a)v\in\ccH.
\]
Then $
\phi_{a,t}=\pi^\circ
(a_t)
v$, where 
$a_t:=\frac{1}{t}
(\lambda_{e^{tx}}(a)-a)-\lambda_x(a)\in\cA^\circ$.
To complete the proof, we need to show that $\lim_{t\to 0}\|\phi_{a,t}\|=0$. But 
\[
\|\phi_{a,t}\|=\|\pi^\circ(a_t)v\|=\|\whpi(a_t)v\|\leq \|a_t\|\cdot\|v\|
\]
and therefore it suffices to show that $\lim_{t\to 0}\|a_t\|=0$. Without loss of generality we can assume that $a=D\otimes f$. From the definition of the norm of $\cA$ we obtain
\[
\|a_t\|=\sup_{(\sigma,\rho^\sigma,\mathscr{K})}
\left\|
\frac{1}{t}\left(
\sigma(e^{tx})\rho^\sigma(D)\sigma(f)-\rho^\sigma(D)
\sigma(f)\right)-\rho^\sigma(xD)\sigma(f)
\right\|,
\] 
where the supremum is taken over all unitary equivalence classes of cyclic unitary representations 
$(\sigma,\rho^\sigma,\mathscr{K})$
of $(G,\g g)$.
Now fix a unitary representation $(\sigma,\rho^\sigma,\mathscr{K})$ and a vector $v\in\ccH_\sigma$ such that 
$\|v\|=1$.  By Taylor's Theorem,
\begin{align*}
\sigma(e^{tx})\rho^\sigma(D)\sigma(f)v&=
\rho^\sigma(D)\sigma(f)v+t\rho^\sigma(xD)\sigma(f)v\\
&+
\frac12\int_0^t(t-s)\sigma(e^{sx})\rho^\sigma(x^2D)\sigma(f)v\,ds.
\end{align*}
By Proposition \ref{fDincAbdd}, there exists a constant $M>0$, independent of $(\sigma,\rho^\sigma,\mathscr K)$, such that 
\[
\|\rho^\sigma(x^2D)\sigma(f)\|\leq M.
\] It follows that
$\|a_t\|\leq \frac{1}{2}M\cdot |t|$, and consequently
$\lim_{t\to 0}\|a_t\|=0$.
\end{proof}

\begin{thm}
\label{thm-isomcat}
The correspondences of
Propositions \ref{prop-pirhotowhpi} and \ref{prp-isunGGg} result in an isomorphism between the category of unitary representations of $(G,\g g)$ and the category of nondegenerate $*$-representations of $\cA=\cA(G,\g g)$.\end{thm}
\begin{proof}
\textbf{Step 1.}
First we  verify that the correspondences of 
Proposition \ref{prop-pirhotowhpi} and Proposition 
\ref{prp-isunGGg}
are mutual inverses. Let $(\pi,\rho^\pi,\ccH)$ be a unitary representation of $(G,\g g)$. Let $\whpi$ be the $*$-representation of $\cA$ constructed by Proposition \ref{prop-pirhotowhpi}, and 
$(\oline\pi,\rho^{\oline \pi},\ccH)$ be the unitary representation of $(G,\g g)$ constructed from $\whpi$ by 
Proposition 
\ref{prp-isunGGg}.
For $D\otimes f\in\cA^\circ$ and $g\in G_\boldeps$,
\begin{align*}
\oline\pi(g)\whpi(D\otimes f)&=\whpi(\lambda_g(D\otimes f))
=\whpi(\alpha_g(D)\otimes \LL_gf)\\
&=
\rho^\pi(\alpha_g(D))\pi(\LL_gf)=\pi(g)\rho^\pi(D)\pi(f)
=\pi(g)\whpi(D\otimes f).
\end{align*}
Since $\pi(g)$ and $\oline\pi(g)$ are bounded operators and
$\whpi$ is nondegenerate, we obtain $\pi(g)=\oline\pi(g)$ for $g\in G_\boldeps$. Let $\ccH^\infty$ denote the space of smooth vectors of $(\pi,\ccH)$.
For $x\in\g g$, $D\otimes f\in\cA^\circ$, and $w\in\ccH$,
\begin{align}
\label{pibarrho=pi}
\rho^{\oline\pi}(x)\whpi(D\otimes f)w
&=
\whpi(xD\otimes f)w=\rho^\pi(xD)\pi(f)w\\ \notag
&=\rho^\pi(x)\rho^\pi(D)\pi(f)w=\rho^\pi(x)\whpi(D\otimes f)w.
\end{align}
By the Dixmier--Malliavin Theorem,  $\ccH^\infty=\whpi(\cA^\circ)\ccH$. Therefore 
\eqref{pibarrho=pi} implies that $\rho^{\oline\pi}(x)=\rho^\pi(x)$
for every $x\in \g g$.

Conversely, let $\whpi:\cA\to B(\ccH)$ be a nondegenerate $*$-representation. Let $(\pi,\rho^\pi,\ccH)$ be the unitary representation of $(G,\g g)$ corresponding to $\whpi$ by Proposition \ref{prp-isunGGg}, and let $\whpi':\cA\to B(\ccH)$ be the 
$*$-representation corresponding to $(\pi,\rho^\pi,\ccH)$ by Proposition \ref{prop-pirhotowhpi}.
For  $D_1\otimes h_1\in\cA^\circ$, and $w\in\ccH$, we obtain 
by Lemma~\ref{lem-gammafDh} that
\begin{align}
\label{eq-auxx}
\pi(f)\whpi(D_1\otimes h_1)w
&=
\int_{G_\boldeps}f(g)\pi(g)\whpi(D_1\otimes h_1)wdg \\
&=
\int_{G_\boldeps} \hat\pi(\gamma_{f, D_1, h_1}(g))w\, dg \notag \\ 
&=\whpi\left(
\int_{G_\boldeps}\gamma_{f,D_1,h_1}^{}(g)dg\right)w=\whpi\big((1\otimes f)(D_1\otimes h_1)\big)w.\notag
\end{align}
Now set $a:=D_1\otimes h_1$. From \eqref{eq-auxx} it follows that for every $D\otimes f\in\cA^\circ$,
\begin{align*}
\whpi'(D\otimes f)\whpi(a)w&=\rho^\pi(D)\pi(f)\whpi(a)w\\
&=\rho^\pi(D)\whpi\big((1\otimes f)a\big)w
=\whpi(D\otimes f)\whpi(a)w.
\end{align*}
Nondegeneracy of $\whpi$ and boundedness of 
the operators 
$\whpi'(D\otimes f)$ and $\whpi(D\otimes f)$
 imply that $\whpi'(D\otimes f)=\whpi(D\otimes f)$. Since $\cA^\circ$ is dense in $\cA$, the equality $\whpi'(a)=\whpi(a)$ holds for every $a\in\cA$.\\

\textbf{Step 2.} 
To complete the proof, we need to show that the correspondences of 
Proposition \ref{prop-pirhotowhpi} and Proposition 
\ref{prp-isunGGg} are compatible with morphisms in the two categories.
Suppose that $(\pi,\rho^\pi,\ccH)$ and $(\sigma,\rho^\sigma,\ccK)$
are two unitary representations of $(G,\g g)$, and  let $\whpi:\cA\to B(\ccH)$ and $\whsi:\cA\to B(\ccK)$ be the  $*$-representations of $\cA$ constructed from
$(\pi,\rho^\pi,\ccH)$ and $(\sigma,\rho^\sigma,\ccK)$
by Proposition \ref{prop-pirhotowhpi}.
 If $T:\ccH\to \ccK$ is a $(G,\g g)$-intertwining operator, then it is easy to verify that $T$ commutes with the action of $\cA^\circ$ on $\ccH$ and $\ccK$, and therefore by a continuity argument, $T$  commutes with the action of $\cA$ on $\ccH$ and $\ccK$ as well.

Conversely, assume that $T:\ccH\to\ccK$ commutes with the actions of $\cA$ on $\ccH$ and $\ccK$. First note that
for every $a\in\cA$ and every $(\lambda,\rho)\in M(\cA)$, 
\begin{align*}
T\whpi\big((\lambda,\rho)\big)\whpi(a)
&=
T\whpi(\lambda(a))=
\whsi(\lambda(a))T=\whsi\big((\lambda,\rho)\big)\whsi(a)T=
\whsi\big((\lambda,\rho)\big)T\whpi(a).
\end{align*}
Since $\whpi(\cA)\ccH$ is a dense subspace of $\ccH$, it follows that
\[
T\whpi\big((\lambda,\rho)\big)=\whsi\big((\lambda,\rho)\big)T\text{ for every }
(\lambda,\rho)\in M(\cA).
\]
Setting $(\lambda,\rho):=(\lambda_g,\rho_g)$ in the last relation, we obtain 
\begin{equation}
\label{TpigsgTT}
T\pi(g)=\sigma(g)T\text{ for every }g\in G_\boldeps,
\end{equation} 
and in particular $T\ccH^\infty\subseteq\ccK^\infty$. 
From \eqref{TpigsgTT} for $g=\boldeps$, it follows that $T$ preserves the $\Z/2\Z$-grading of $\ccH$. 
Now for $(\lambda,\rho)\in M(\cA^\circ)$, $a\in\cA^\circ$, and $v\in\ccH^\infty$, using Lemma 
\ref{piissmoothrp} we obtain that 
\begin{align*}
T\pi^\circ(\lambda,\rho)\pi^\circ(a)v&=T\pi^\circ(\lambda(a))v
=
\sigma^\circ(\lambda(a))Tv\\ &
=
\sigma^\circ(\lambda,\rho)\sigma^\circ(a)Tv
=
\sigma^\circ(\lambda,\rho)T\pi^\circ(a)v.
\end{align*}
Thus  
Lemma 
\ref{piissmoothrp}
and 
Lemma \ref{DMAA=A} imply that $T\pi^\circ(\lambda,\rho)w=\sigma^\circ(\lambda,\rho)Tw$ for every $w\in\ccH^\infty$. Setting $(\lambda,\rho):=(\lambda_x,\rho_x)$ for $x\in\g g$, we obtain 
$T\rho^\pi(x)=\rho^\sigma(x)T$. Therefore $T$ is a $(G,\g g)$-intertwining map 
from $(\pi,\rho^\pi,\ccH)$ to $(\sigma,\rho^\sigma,\ccK)$.
\end{proof}

\section{Unique direct integral decompositions}
\label{sec-unique}
For a unitary representation $(\pi,\rho^\pi,\ccH)$ of a Lie supergroup $(G,\g g)$, it is desirable to have a decomposition as a direct integral of irreducible unitary representations. From Theorem \ref{thm-isomcat} it follows that the problem of existence and uniqueness of such a direct integral decomposition can be reduced to the same problem for the associated $C^*$-algebra $\cA=\cA(G,\g g)$. 

In this section we prove that 
 existence and uniqueness of direct integral decompositions hold for two general classes of Lie supergroups, which include nilpotent and basic classical Lie supergroups.

Recall that 
a $C^*$-algebra $\cA$ is called \emph{CCR} if 
$\whpi(\cA)\sseq K(\ccH)$ 
for every 
irreducible $*$-representation $\whpi:\cA\to B(\ccH)$,
where $K(\ccH)\sseq B(\ccH)$ denotes the subspace of compact operators. It is well known that for $C^*$-algebras which are CCR, existence and uniqueness of direct integral decompositions hold.

A unitary representation $(\pi,\ccH)$ of a Lie group $G$ is called \emph{completely continuous} if 
$\pi(f)\in K(\ccH)$ for every $f\in\cD(G)$.

\begin{thm}
\label{thmCCR}
Let $(G,\g g)$ be a Lie supergroup such that
for every irreducible unitary representation $(\pi,\rho^\pi,\ccH)$ of $(G,\g g)$, the unitary representation $(\pi,\ccH)$ of $G$ is
completely continuous.
Then the $C^*$-algebra $\cA=\cA(G,\g g)$ is CCR. 
\end{thm}

\begin{proof}
Let $\whpi:\cA\to B(\ccH)$ be an irreducible 
$*$-representation of $\cA$. Since $K(\ccH)$
is a closed ideal of $B(\ccH)$ and 
$\|\whpi(a)\|\leq\|a\|$ for every $a\in\cA$,  
it suffices to prove that $\whpi(D\otimes f)\in K(\ccH)$ for every $D\otimes f\in \cA^\circ$. 
Let $(\pi,\rho^\pi,\ccH)$ be the unitary representation of  $(G,\g g)$ that corresponds to $\whpi$. 
Theorem \ref{thm-isomcat} implies that 
$(\pi,\rho^\pi,\ccH)$
is irreducible.
The Dixmier--Malliavin Theorem implies that there exist 
$f_1,\cdots,f_r,h_1,\ldots,h_r
\in\cD(G_\boldeps)$ such that
$f=\sum_{i=1}^rf_i\star h_i$.
Thus 
\[
\whpi(D\otimes f)=\sum_{i=1}^r
\whpi(D\otimes f_i)\whpi(1\otimes h_i)
=
\sum_{i=1}^r\whpi(D\otimes f_i)\pi(h_i).
\] 
From the assumption of the theorem it follows that  $\pi(h_i)\in K(\ccH)$ for  $1\leq i\leq r$. Consequently, $\whpi(D\otimes f)\in K(\ccH)$.
\end{proof}

As in  
\cite{SalCMP} or \cite[Sec.\ 8]{nsLNM},
a Lie supergroup $(G,\g g)$ is called \emph{nilpotent} if $\g g$ is a nilpotent Lie superalgebra.
\begin{thm}
Let $(G,\g g)$ be a connected nilpotent Lie supergroup. Then the $C^*$-algebra $\cA=\cA(G,\g g)$ is CCR.
\end{thm}
\begin{proof}
From \cite[Cor. 6.1.1]{SalCMP} it follows that the restriction of every irreducible unitary representation $(\pi,\rho^\pi,\ccH)$ of $(G,\g g)$ to $G$ is a direct sum of finitely many irreducible unitary representations. 
Since every nilpotent Lie group is CCR \cite{Fell}, the unitary representation $(\pi,\ccH)$ is completely continuous. 
Therefore  Theorem \ref{thmCCR} applies. 
\end{proof}

Recall from \cite{nsLNM} that a Lie supergroup $(G,\g g)$ is called \emph{$\star$-reduced} if for every nonzero $x\in\g g$ there exists a unitary representation $(\pi,\rho^\pi,\ccH)$ of $(G,\g g)$ such that $\rho^\pi(x)\neq 0$.

\begin{thm}
\label{thm6.3}
Let $(G,\g g)$ be a connected Lie supergroup 
satisfying $\g g_\eev=[\g g_\ood,\g g_\ood]$.
Then the $C^*$-algebra $\cA=\cA(G,\g g)$ is CCR.
\end{thm}
\begin{proof}
We show that the hypotheses of Theorem \ref{thmCCR}
are satisfied.\\

\noindent\textbf{Step 1.} Let $\fn \subseteq \fg$ be the intersection of the 
kernels of all representations $\rho^\pi$, where 
$(\pi,\rho^\pi,\ccH)$ is a unitary representation of $(G,\fg)$. Then 
$\fn$ is a superideal and its even part is the Lie algebra 
of a closed normal subgroup $N \subseteq G$. Passing to 
$(G/N,\fg/\fn)$, we may therefore assume w.l.o.g.\ that $\fg$ is $\star$-reduced. 
Let $(\pi,\rho^\pi,\ccH)$ be an irreducible unitary representation of $(G,\fg)$. 

From \cite[Thm 7.3.2]{nsLNM} it follows that 
there exists a compactly embedded (in the sense of \cite[Def.\ VII.1.1]{nHolo})
Cartan subalgebra $\g t\subseteq\g g_\eev$ and a positive system $\Delta^+=\{\alpha_1,\ldots,\alpha_r\}$ of $\g t$-roots of $\g g$, such that the space $\ccH^{[\g t]}$ of $\g t$--finite smooth vectors in $\ccH$ is a dense subspace of $\ccH$. Furthermore, $\ccH^{[\g t]}$ is an irreducible $\g g$-module which is a direct sum of $\g t$-weight spaces with weights of the form \begin{equation}
\label{hwofH}
\lambda-\sum_{i=1}^rn_i\alpha_i
\text{\ where $n_i\in\N\cup\{0\}$ for every $1\leq i\leq r$.}
\end{equation}  
Since $\bfU(\g g)$ is a finitely generated $\bfU(\g g_\eev)$-module, the irreducible (hence cyclic) $\bfU(\g g)$-module $\ccH^{[\g t]}$ is a finitely generated $\bfU(\g g_\eev)$-module. 
 Since 
$\bfU(\g g_\eev)$ is a Noetherian ring
\cite[Cor. 2.3.8]{dixmier},  
it follows that $\ccH^{[\g t]}$ is a Noetherian 
$\bfU(\g g_\eev)$-module.\\

\noindent\textbf{Step 2.} We prove that 
$(\pi,\ccH)$ is a direct sum of finitely 
many irreducible unitary representations of $G$. 
Assume the contrary. 
Then we can write 
$\ccH=\bigoplus_{\ell=1}^\infty\ccH_{\ell}$ such that each 
$\ccH_{\ell}$ is a $G$-invariant subspace of $\ccH$.
From the inclusion $\bigoplus_{i=1}^\infty\ccH_{\ell}^{[\g t]}\subseteq \ccH^{[\g t]}_{}$ it follows that as a 
$\bfU(\g g_\eev)$-module,
$\ccH^{[\g t]}$ is not Noetherian. 
This contradicts Step 1.\\

\noindent\textbf{Step 3.} From \eqref{hwofH} and Step 2 it follows that $(\pi,\ccH)$ is a direct sum of finitely many irreducible highest weight unitary representations of $G$ 
(in the sense of \cite[Def.\ X.2.9]{nHolo}). From 
\cite[Thm X.4.10]{nHolo} it follows that every irreducible highest weight unitary representation of $G$ is CCR. Thus $(\pi,\ccH)$ is also a CCR unitary representation of $G$.
\end{proof}

\begin{rmk}
Let $(G,\g g)$ be a Lie supergroup such that 
$\g g$ is a real form of a classical
simple  Lie superalgebra (see \cite[Sec.\ 1.3]{musson}). That is, we assume that $\g g\otimes_\R^{}\C$ is isomorphic to one of the Lie superalgebras of type $\g{sl}(m|n)$ for $m>n\geq 0$, $\g{psl}(m|m)$ for $m\geq 1$, 
$\g{osp}(m|2n)$ for $m,n\geq 0$, $D(2,1;\alpha)$ for $\alpha\neq0,-1$, 
$\g{p}(n)$ for $n\geq 1$, $\g{q}(n)$ for $n\geq 1$, $G(3)$, or $F(4)$. 
Assume that $(G,\g g)$ has nontrivial unitary representations. (A complete list of these Lie supergroups can be obtained from \cite[Thm 6.2.1]{nsLNM}.) It is then straightforward to verify that $(G,\g g)$ satisfies the hypotheses of Theorem \ref{thm6.3}, and therefore the $C^*$-algebra $\cA=\cA(G,\g g)$ is CCR.

\end{rmk}

\end{document}